\documentclass[12pt]{amsart}
\usepackage{amsmath}
\usepackage{amssymb}
\usepackage{latexsym}
\def\ds{\displaystyle}
\newtheorem{theorem}{Theorem}[section]

\newtheorem{lemma}[theorem]{Lemma}

\newtheorem{prop}[theorem]{Proposition}
\newtheorem{defn}[theorem]{Definition}
\newtheorem{rem}[theorem]{Remark}
\newenvironment{proof*}{\vskip 2mm\noindent {}}{\hfill $\Box$ \vskip 2mm}
\numberwithin{equation}{section}
\newcommand{\C}{{\mathbb{C}}}
\newcommand{\D}{{\mathbb{D}}}
\newcommand{\G}{{\mathbb{G}}}
\newcommand{\Z}{{\mathbb{Z}}}

\renewcommand{\O}{{\mathcal{O}}}
\newcommand{\eps}{\varepsilon}
\newcommand{\la}{\lambda}

\begin{document}

\title[Green function of the spectral ball]{Green functions of the spectral ball and symmetrized polydisk}

\author{P. J. Thomas, N. V. Trao, W. Zwonek}

\address{Universit\'e de Toulouse\\ UPS, INSA, UT1, UTM \\
Institut de Math\'ematiques de Toulouse\\
F-31062 Toulouse, France}
\email{pthomas@math.univ-toulouse.fr}

\address{ Department of Mathematics\\
Hanoi National University of Education\\
136 Xuan Thuy str - Cau Giay\\
Hanoi - Vietnam}
\email{ngvtrao@yahoo.com}

\address{Instytut Matematyki, Uniwersytet Jagiello\'nski, \L ojasiewicza 6, 30-348
 Krak\'ow, Poland}\email{Wlodzimierz.Zwonek@im.uj.edu.pl}

\begin{thanks}{This research has been supported in part by the Grant
no. N N201 361436 of the Polish Ministry for Higher Education.
Work on this paper was started during the
stay of the first named author at the Jagiellonian University, Cracow.}
\end{thanks}

\begin{abstract}
The Green function of the spectral ball is constant over the isospectral
varieties, is never less than the pullback of its counterpart
on the symmetrized polydisk, and is equal to it in the generic case
where the pole is a cyclic (non-derogatory) matrix.  When the pole is
derogatory, the inequality is always strict, and the difference
between the two functions depends on the order of
nilpotence of the strictly upper triangular
blocks that appear in the Jordan decomposition of the pole.  In particular,
the Green function of the spectral ball is not symmetric in its
arguments. Additionally, some estimates are given
for invariant functions in the symmetrized polydisc, e.g. (infinitesimal versions of) the Carath\'eodory distance and the Green function, that show that they are distinct in dimension greater
or equal to $3$.
\end{abstract}

\maketitle

\section{Introduction and statement of results}
\label{intro}

Let $\mathcal M_n$ be the set of all $n\times n$ complex matrices.
For $A\in\mathcal M_n$ denote by $sp(A)$ and $\ds
\rho(A)=\max_{\lambda\in sp(A)}|\lambda|$ the spectrum and the
spectral radius of $A,$ respectively. The notation $\|A\|$
will stand for an operator norm on the set of matrices
(chosen once and for all).

The \emph{spectral ball} $\Omega_n$
is the set
$$\Omega_n=\{A\in\mathcal M_n:\rho(A)<1\}.$$
The
characteristic polynomial of the matrix $A$ is denoted
$$
P_A(t):= \det(tI-A)=: t^n+\sum_{j=1}^n(-1)^j\sigma_j (A)t^{n-j},
$$
where $I\in \mathcal M_n$ is the unit matrix. We define
a map $\sigma$ from $\mathcal M_n$ to $\mathbb C^n$ by
$\sigma := (\sigma_1, \dots, \sigma_n)$.
 The \emph{symmetrized polydisk}
is $\mathbb G_n := \sigma (\Omega_n)$ is a bounded domain in $\mathbb C^n$,
which is a complete hyperbolic domain, and hyperconvex (and thus taut).

A matrix $A$ is \emph{cyclic} (or \emph{non-derogatory})
if it admits a cyclic vector, we then write $A \in \mathcal C_n$.
We say that $A$ is derogatory when $A \notin \mathcal C_n$.

\begin{defn}
The Green function with pole $p$ in a domain $\Omega$
is given by
$$
g_\Omega (p,z) := \sup \{ u(z) : u \in PSH_-(\Omega),
u(w) \le \log \|w-p\| + O(1) \}.
$$
\end{defn}

Let $\D$ stand for the unit disk in $\C$.

\begin{defn}
The Lempert function of a domain $D\subset\Bbb C^m$ is defined, for $z,w\in D$, as
$$
 l_D(z,w):=\inf\{|\alpha|: \alpha \in \mathbb D
\mbox{ and } \exists\varphi\in\mathcal O(\mathbb
D,D):\varphi(0)=z,\varphi(\alpha)=w\}.
$$
\end{defn}

\begin{defn}
The Carath\'eodory (pseudo)distance for a domain $D\subset\Bbb C^m$ is defined, for $w,z\in D$, as
$$
c_D^*(z,w):=\sup\{|f(w)|: f\in\mathcal O(D,\mathbb D),f(z)=0\}.
$$
\end{defn}

Immediate consequences of the definitions are that for any domain $D$
in $\C^n$,
$$
\log c_D^*(z,w)\le  g_D(z,w) \le \log l_D(z,w), \mbox{ and }
$$
$$
g_{\Omega_n}(V,M) \ge g_{\G_n} (\sigma (V), \sigma (M)).
$$
One can prove that $\log l_{\Omega_n}(0,M)=g_{\Omega_n}(0,M) = \log \rho(M)$. This follows
from Vesentini's theorem about the plurisubharmonicity of $\log \rho$ \cite{Ve}
and the facts that $\rho(\lambda A)=|\lambda| \rho( A)$, for $\lambda \in \mathbb C$
(see also \cite[Theorem 3.4.7, p. 52]{Au} and \cite{Jar-Pfl 1993}).


As is noted in
\cite{Edi-Zwo}, $\sigma(A)=\sigma(B)$ if and only if
there is an entire curve contained in $\Omega_n$ going through $A$ and $B$.
It follows from Liouville's theorem for subharmonic
functions that if $ \sigma (M)= \sigma (M')$, then
$g_{\Omega_n}(V,M)=g_{\Omega_n}(V,M')$.
So $g_{\Omega_n}(V,M)$ only depends on $\sigma(M)$.
One may wonder, then, whether for any $V,M$,
$$
g_{\Omega_n}(V,M) = g_{\G_n} (\sigma (V), \sigma (M)) ?
$$
We will prove this only happens  when $V \in \mathcal C_n$.

Let us proceed with some elementary reductions.
For any $Q \in \mathcal M_n^{-1}$ (the set of invertible matrices), the map
$M \mapsto Q^{-1} M Q$ is an automorphism of the spectral
ball preserving the spectrum, so
$$
g_{\Omega_n}(Q^{-1} V Q,M)=g_{\Omega_n}(V,Q M Q^{-1})=g_{\Omega_n}(V,M),
$$
thus we may always assume that our pole matrix $V$ is in Jordan form
(or any other convenient reduction by conjugation).

For any $\lambda \in Sp(V)$, denote
by $V_\lambda$ the restriction of $V$ to the stable subspace
$\ker (V-\la I_n)^n$. Let $n(\la):= \dim(\ker (V-\la I_n)^n)$
(the size of the Jordan block corresponding
to the eigenvalue $\lambda$)
and $m(\lambda):= \min \{k: (V_\lambda-\la I_{n(\la)})^k =0\}$  the order of nilpotence of $V_\lambda-\la I_{n(\la)}$.
Finally there exists $\lambda \in Sp(V)$ such that
 $m(\la)<n(\la)$
if and only if $V \notin \mathcal C_n$.

\begin{theorem}
\label{diff}
Let $V \in \Omega_n$.
\begin{enumerate}
\item
If $V \in \mathcal C_n$, then
$g_{\Omega_n}(V,M) = g_{\G_n} (\sigma (V), \sigma (M)).$
\item
If
$V \notin \mathcal C_n$,
then
there exists $X \in \mathcal M_n \setminus \{0\}$
such that
\begin{eqnarray}
\label{upest}
g_{\Omega_n}(V,V+\zeta X) & \ge & m(\lambda) \log |\zeta| + O(1), \mbox{ while } \\
\label{downest}
g_{\G_n} (\sigma (V), \sigma (V+\zeta X)) & \le & n(\lambda) \log |\zeta| + O(1).
\end{eqnarray}
\end{enumerate}
\end{theorem}

\begin{proof}

Part (1)
follows from a theorem of Jarnicki and Pflug \cite[Theorem 1]{JP1},
because the rank of the differential of $\sigma$
at $A$ is maximal precisely
when $A \in \mathcal C_n$ \cite{NiThZw}. Part (2) will be proved in sections \ref{pfnil} and \ref{pfgen}
below.
\end{proof}

The following result
should be compared with \cite[Theorem 1.3]{ThTr}, which states
that the continuity at $A$
of $l_{\Omega_n}(.,M)$, for any $M \in \Omega_n$, implies
cyclicity of $A$ (with the converse holding for $n\le 3$,
see \cite[Proposition 1.4]{ThTr}).

\begin{prop}
\label{contcrit}
Let $A, M \in \Omega_n$.  The following properties are equivalent:
\begin{enumerate}
\item
$g_{\Omega_n}(A,M)=g_{\mathbb G_n}(\sigma(A),\sigma(M)).$
\item
The Green function $g_{\Omega_n}$ is continuous at $(A,M)$.
\item
The function $g_{\Omega_n}(.,M)$ is continuous at $A$.
\end{enumerate}
\end{prop}

An immediate corollary of Theorem \ref{diff} is that the function $g_{\Omega_n}$ is not symmetric in
its arguments. Recall that both the Lempert function and the Carath\'eodory distance are symmetric (for all domains). Since $g_{\G_2}=\log l_{\G_2}=\log c_{\G_2}^*$ (\cite{Agl-You}, \cite{Cos}) the Green function $g_{\G_2}$ is symmetric.
We conjecture that $ g_{\G_n}$ fails to be symmetric for $n\ge 3$.

Even though we cannot prove the above conjecture, we are able to get some estimates
between (logarithm of) the Carath\'eodory distance and the Green function in the symmetrized polydisc, showing in particular that
these two objects differ in $\G_n$, $n\ge 3$, which extends some of the results from \cite{Nik-Pfl-Tho-Zwo}. We get this from facts about
their infinitesimal versions.  Recall that the \emph{Carath\'eodory-Reiffen} and \emph{Azukawa} pseudometrics in a domain $D\subset \C^n$ are respectively given by
\begin{multline*}
\gamma_D (z, X):=\sup \left\{ \left| f'(z) \cdot X\right| : f \in \mathcal O (D,\D),
f(z)=0 \right\}, \\
A_D (z, X) := \limsup_{\lambda\to 0}\frac{\exp g_D(z,z+\lambda X)}{|\lambda| },
\mbox{ for } z \in D, X \in \C^n.
\end{multline*}
Recall that one may replace '$\limsup$' in the definition of the Azukawa metric above with '$\lim$' when $D$ is a bounded hyperconvex domain (in particular, when $D=\G_n$) -- see e. g. \cite{Zwo}. We also make use of the fact that $\gamma_D(z,X)=\lim_{\lambda\to 0}\frac{c_D^{*}(z,z+\lambda X)}{|\lambda|}$ (see e. g. \cite{Jar-Pfl 1993}).

\begin{theorem}
\label{carazu}
For $n\ge 3$,
$
\gamma_{\G_n}(0;e_{n-1})<  A_{\G_n}(0;e_{n-1}),
$
and consequently
$c^*_{\G_n}(0,te_{n-1}) < \exp g_{\G_n}(0,te_{n-1}) $  for $|t|$ small enough.
\end{theorem}
This follows from Proposition \ref{gamgn}. The explicit estimates in Section~\ref{symmpol} show that holomorphically invariant objects differ very much in $\G_n$, $n\ge 3$, in sharp contrast to the case $n=2$.

\section{Proof of Proposition  \ref{contcrit}}
\label{pfprop}

That (2) implies (3) is clear.

\noindent{Proof of (3) $\Rightarrow$ (1).}

Since the cyclic matrices are dense in
$\Omega_n$ then there exist  $A_j \in
\mathcal{C}_n$ such that $A_j \to A.$
By continuity of  $g_{\Omega_n}(\cdot,M)$ at $A$, we
get that $g_{\Omega_n}(A_j,M) \xrightarrow{j
\to \infty} g_{\Omega_n}(A,M).$

On the other hand, by Theorem \ref{diff}(1), $g_{\Omega_n}(A_j,M) g_{\mathbb{G}_n}(\sigma(A_j),\sigma(M)).$
By hyperconvexity of domain $\mathbb{G}_n$ we have
$g_{\mathbb{G}_n}(\sigma(A_j),\sigma(M))
\xrightarrow{j \to \infty}
g_{\mathbb{G}_n}(\sigma(A),\sigma(M)).$
This implies that $g_{\Omega_n}(A,M) = g_{\mathbb{G}_n}(\sigma(A),\sigma(M)).$

\noindent{Proof of (1) $\Rightarrow$ (2).}

Assume $g_{\Omega_n}(A,M) = g_{\mathbb{G}_n}(\sigma(A),\sigma(M))$.

Let $(A_j,M_j)\subset \Omega_n$ be such that
$(A_j,M_j) \xrightarrow{j\to \infty} (A,M)$ and
$$
\lim_{j\to \infty}g_{\Omega_n}(A_j,M_j) = a :\liminf_{(X,Y)\to (A,M)} g_{\Omega_n}(X,Y).
$$
We have
$$
g_{\Omega_n}(A_j,M_j) \ge
g_{\mathbb{G}_n}(\sigma(A_j),\sigma(M_j)) \to
g_{\mathbb{G}_n}(\sigma(A),\sigma(M)),
$$
and hence $a \geq g_{\mathbb{G}_n}(\sigma(A),\sigma(M)) = g_{\Omega_n}(A,M).$
Then $g_{\Omega_n}$ is lower semicontinuous at $(A,M).$
Since  $g_{\Omega_n}$ is upper
semicontinous \cite{JP2}, it is continuous at $(A,M)$.

\section{Proof of Theorem \ref{diff}(2): the nilpotent case}
\label{pfnil}


When we make the additional assumption that $V$ is nilpotent, equivalently
$Sp(V)=\{0\}$, we have $n(\lambda)=n$, $m(\lambda)=m:= \min \{k: V^k =0\}$,
the order of nilpotence of $V$.

We begin by proving \eqref{upest}.

\begin{lemma}
\label{down}
Let $V, m$ be as in the hypotheses of Theorem \ref{diff} (2).
Then $\log \rho (V+A) \le \frac1m \log \|A\| +O(1)$, and as a consequence
$g_{\Omega_n}(V,M) \ge m \log \rho (M)$, for any $M\in \Omega_n$.
\end{lemma}

\begin{proof}
We assume that $V=(v_{ij})_{1\le i,j\le n}$ is in Jordan form with the following notations.
Let $r$
stand for the rank of $V$.
Write $$
F_0:=\{ j : v_{ij}=0\mbox{ for  }1\le i \le n \}
:=\{ 1=b_1 < b_2 < \dots < b_{n-r} \}.
$$
For
all the other values of $j$, $v_{j-1,j}=1$, $v_{ij}=0$ for $i\neq j-1$.
We can choose the Jordan form so that $b_{l+1}-b_l$ is decreasing for $1\le l \le n-r$,
with the convention $b_{n-r+1}:=n+1$.
With this choice of notation (and order), $m=b_2 - b_1$.

Now we must study the homogeneity of the functions $ \sigma_i (V+A)$ in terms of the
entries of $A$.  This is Lemma 4.2 from \cite{ThTr}.

\begin{lemma}
\label{degree}
Let $d_i := 1+ \# \left( F_0 \cap [(n-i+2)..n] \right) $.
The lowest order terms of  $\sigma_i (V+A) $
are of degree $d_i$ (in the entries of $A$).
\end{lemma}

Then the eigenvalues $(\la_1, \dots, \la_n)$ of $V+A $ satisfy the following equations:
$$
s_i(\la_1, \dots, \la_n) = \sigma_i (V+A), 1\le i \le n,
$$
where
$s_i(\la_1, \dots, \la_n)$
stands for the elementary symmetric function of degree $i$.

\begin{lemma}
\label{estdi}
$m d_i \ge i$, for  $1\le i \le n$.
\end{lemma}
Then we set $\la':= \la \|A\|^{-1/m}$, and we have the new equations
(for $A\neq 0$)
$$
s_i(\la'_1, \dots, \la'_n) = \sigma_i (V+A) \|A\|^{-i/m}, 1\le i \le n,
$$
and by the Lemma the right hand sides are bounded functions of $A$
near $0$.  Since a polynomial of the form $X^n + \sum_j \alpha_j X^j$
where $|\alpha_j | \le C$ has all its roots in a disk of radius
$C n^{1/n}$ about the origin, all the solutions of those equations
are bounded by a constant (which depends on $V$), thus
$\la = O(\|A\|^{1/m})$. Taking logarithms, we find the desired
estimate on $u$.

\begin{proof*}{{\it Proof of Lemma \ref{estdi}.}}

Suppose that $b_l \le i<b_{l+1}$.  Then $d_i = l$, so it will be
enough to prove that $m l \ge b_{l+1} -1$, for any $l\le n-r$. But,
by our hypothesis of decrease of the $b_{j+1}-b_j$,
$$
b_{l+1} -1 = \sum_{j=1}^l b_{j+1}-b_j
\le \sum_{j=1}^l b_2-b_1 = l m .
$$
\end{proof*}
\end{proof}

{\bf Remark.}

The bound in Lemma \ref{down} is optimal.  Indeed, recall
that $m=b_2-b_1=b_2-1$. Let $X:=(x_{ij})$ where
$x_{m1}=1, x_{ij} =0$ otherwise. Then
$$
P_{V+\zeta X}(t) = (t^m - \zeta)t^{n-m},
$$
so $\rho (V+\zeta X) = |\zeta|^{1/m}$. The map $\psi(\zeta)=V+\zeta X$
sends $\D$ to $\Omega_n$, so $l_{\Omega_n}(V, V+\zeta X)\le |\zeta|$, so
$$
g_{\Omega_n}(V, V+\zeta X)
\le
\log l_{\Omega_n}(V, V+\zeta X)\le \log |\zeta|
= m \log \rho ( V+\zeta X).
$$
\vskip.3cm

To prove \eqref{downest}, choose a matrix $X$
with $\sigma_i(X)=0$ for $i\le n-1$, $\sigma_n(X)=(-1)^{n-1}$
(the spectrum is then made up of all the $n$-th roots of unity).
Then
$$
\sigma (\zeta X) = \zeta^{n} (0, \dots, 0,
\sigma_{n}(X) ).
$$
Therefore
$g_{\G^n} (0, \sigma (\zeta X) )\le n \log |\zeta| + O(1).$
\hfill \qed

To see more general cases of matrices $X$ where the
Green function of the spectral ball is strictly above the
pull back $g_{\G^n} \circ \sigma$, take $X$
such that its characteristic polynomial verifies
$\sigma_i(X)=0$ for $i\le m$, and that its eigenvalues
are all distinct and nonzero. This is always
possible, since $m\le n-1$.
Then $g_{\G^n} (0, \sigma (\zeta X) )\le (m+1) \log |\zeta| + O(1)
\le   (m+1) \log \rho (V + \zeta X ) + O(1)
< g_{\Omega_n}(0,\zeta X)$ for $\zeta$ small enough.

\section{Proof of the Theorem: general case}
\label{pfgen}
Let $\la_0$ be an eigenvalue such that $m(\la_0):=m_0< n(\la_0)=:n_0$.
 By applying the automorphism
$M \mapsto (\la_0 I_n -M)(I_n - \la_0 M)^{-1}$, we
may reduce ourselves to the case $\la_0=0$, and we may assume further
that
$$
V = \left( \begin{array}{cc} V_0 & 0 \\ 0 & V_1 \end{array} \right),
$$
where $V_0 \in \mathcal M_{n_0}$ is in Jordan form.

\begin{lemma}
\label{splitpoly}
There exist a neighborhood $\mathcal U$ of $\sigma(V)$ in $\G_n$
and $\sigma^0$ a holomorphic map from $\sigma^{-1} (\mathcal U)$
to $\C^{n_0}$ such that
$$
X^{n_0} + \sum_{j=1}^{n_0} (-1)^j \sigma^0_j (M) X^{n_0-j} := P^0_M (X)
= (X-\la_1) \cdots (X-\la_{n_0}),
$$
where $\{ \la_1, \dots, \la_{n_0} \}$ are the smallest $n_0$ eigenvalues
of $M$ (in modulus).
\end{lemma}

\begin{proof}
This fact relies on the holomorphic dependency of a subset of the roots
of a polynomial in a neighborhood of a multiple root, in the spirit of
the Weierstrass Preparation Theorem.

In more detail: for $s=(s_1, \dots, s_n) \in \G_n$, let
$P_s (X) = X^{n} + \sum_{j=1}^{n} (-1)^j s_j X^{n-j} $.
There exists $\delta >0$ such that the open set
$$
\mathcal U_\delta :=\left\{ s\in \G_n :
\#( P_s^{-1}\{0\} \cap D(0,\delta)) = n_0,
P_s^{-1}\{0\} \cap \partial D(0,\delta)=\emptyset  \right\},
$$
where the zeroes are counted with multiplicities,
contains $\sigma(V)$. On $\sigma^{-1} (\mathcal U_\delta)$, the formulas
$$
\Sigma_k (M) := \frac1{2\pi i} \int_{\partial D(0,\delta)} \zeta^k \frac{(P^0)'_M(\zeta)}{P_M(\zeta)} d\zeta
$$
give holomorphic functions which are equal to $\la_1^k + \cdots +\la_{n_0}^k$, and
the elementary symmetric functions of that subset of eigenvalues can be algebraically
deduced from those.
\end{proof}

Notice that the above lemma gives a holomorphically varying factorization of the
characteristic polynomial of $M$ : $P_M(X) = P_M^0 (X) P_M^1 (X) $, and
a holomorphically varying splitting of the space $\C^n$,
$$
\C^n = \ker P_M^0 (M)
\oplus \ker P_M^1 (M) =: U^0\oplus U^1.
$$
Then $P_M^0 =P_{M|_{U^0}}$ and $\rho^0(M):= \rho(M|_{U^0})$ is the largest
modulus of the eigenvalues of $M$ contained in $D(0,\delta)$. So $u(M):= \log \rho^0(M)$
defines a plurisubharmonic function in a neighborhood of $V$.

We follow the scheme of proof of the special case.

Since $g_{\G_n}(\sigma(V),.)=-\infty$ precisely at the point $\sigma(V)$
and $g_{\Omega_n}(V,M) \ge g_{\G_n}(\sigma(V),\sigma(M))$, we can pick an
$\eps_0 >0$ such that
$$
\mathcal U_0 := \sigma\left( \{  g_{\Omega_n}(V,.)< \log \eps_0 \} \right) \subset \mathcal U_\delta.
$$
Therefore
$$
 \sigma^{-1} (\mathcal U_0) = \{  g_{\Omega_n}(V,.)< \log \eps_0 \}
\subset  \sigma^{-1} (\mathcal U_\delta)
$$
(recall that $g_{\Omega_n}(V,.)$ is constant on the fibers of $\sigma$). It is a standard
fact that then
$$
g_{\sigma^{-1} (\mathcal U_0)}(V,.) = g_{\Omega_n}(V,.) - \log \eps_0 .
$$
To compare this local Green function with our function $u$, it is enough
to estimate $u$ near the pole $V$.

\begin{lemma}
\label{uest}
There exists a neighborhood $\mathcal V$ of $0$ in $\mathcal M_n$ such
that for any $A \in \mathcal V$,
$u(V+A) \le \frac1{m_0} \log \|A\| +O(1)$,
and therefore $g_{\sigma^{-1} (\mathcal U_0)}(V,M) \ge m_0 u(M).$
\end{lemma}

This will conclude the proof, since
we can find a matrix $X$
(work as before, but only on the upper left block) such that
$g_{\G_n}(\sigma(V),\sigma(V+\zeta X)) \le n_0 u(V+\zeta X) +O(1)$.
\hfill \qed

\begin{proof*}{{\it Proof of Lemma \ref{uest}}}

For $A$ small enough, $\ker P^0_{V+A} (V+A)$ (respectively
$\ker P^1_{V+A} (V+A)$) is close enough to $\ker P^0_V(V) = \C^{n_0} \times \{0\}$
(respectively to $\ker P^1_V(V) = \{0\} \times\C^{n-n_0} $) so that
the projections $\pi_j$ from $\C^n$ to $\ker P^j_V(V)$
with kernel equal
to $\ker P^{1-j}_{V+A}(V+A) $  ($j=0,1$)
induce bijections from $\ker P^{j}_{V+A}(V+A) $
onto $\ker P^j_V(V)$.

Let $P$ be the matrix of the bijective endomorphism defined by
\newline
$\pi_0|_{\ker P^{0}_{V+A}(V+A)} + \pi_1|_{\ker P^{1}_{V+A}(V+A)} $. Then
$$
P M P^{-1} = \left( \begin{array}{cc} M_0 & 0 \\ 0 & M_1 \end{array} \right),
$$
for some $M_0 \in \mathcal M_{n_0}$ and $M_1 \in \mathcal M_{n-n_0}$.
We have seen that $\{ \la_1, \dots, \la_{n_0} \} = Sp M_0$,
and one can check that $P = I_n + O(\|A\|)$, so that $M_0 = V_0 + O(\|A\|)$.
Applying the proof in Section \ref{pfnil}, $\la_j = O(\|M_0-V_0\|^{1/m_0}) = O(\|A\|^{1/m_0})$,
for $1\le j \le n_0$. The estimate follows easily.
\end{proof*}

\section{Estimates between the Green function and the Carath\'eodory distance in $\G_n$, $n\ge 3$}\label{symmpol}
This part of the paper may be seen as a continuation and extension of the results from \cite{Nik-Pfl-Tho-Zwo}.  Recall \cite{Jar-Pfl 1993} that for any $k \in \Z_+^*$,
$$
\gamma_D^{(k)}(z,X) := \sup\left\{
\limsup_{\lambda\to 0} \frac{|f(z+\lambda X)|^{1/k}}{|\lambda|} ,
f \in \mathcal O(D,\D), \mbox{ord}_z f \ge k
\right\},
$$
and that
$\kappa_D(z,X) \ge A_D(z,X) \ge \gamma_D^{(k)}(z,X) \ge \gamma_{D} (z,X)$.

The definitions and basic properties of some additional infinitesimal functions used below (Kobayashi-Royden metric $\kappa_D$ and Kobayashi-Buseman metric $\hat \kappa_D$) may be found in \cite{Nik-Pfl-Tho-Zwo} or \cite{Jar-Pfl 1993}, with identical notations.

\begin{prop}
\label{low_est} For any $n\ge 2$ the following inequalities hold
$$
\kappa_{\G_n}(0;e_{n-1})\geq A_{\G_n}(0;e_{n-1})\geq \gamma_{\G_n}^{(n-1)}(0;e_{n-1})\geq\root{n-1}\of{(n-1)/n}.
$$
\end{prop}

\begin{proof} We only need to prove the last inequality.

Recall that $\G_n=\pi(\D^n)$, where, with the notation of Section \ref{pfnil} for the elementary symmetric functions,
$$\pi_j(\lambda_1,\ldots,\lambda_n):\left( s_1(\lambda_1,\ldots,\lambda_n), s_2(\lambda_1,\ldots,\lambda_n), \dots,
s_n(\lambda_1,\ldots,\lambda_n) \right).
$$
Consider the function $f(\lambda_1,\ldots,\lambda_n):=(\lambda_1^l+\ldots+\lambda_n^l)/n$, $\lambda_j\in\D$.
We may treat $f$ as a function from $\O(\G_n,\D)$. Recall that it is a polynomial. To get the lower estimate for the Azukawa metric at $0$ in direction $e_{n-1}$ we want the function $f$ to be the function of multiplicity at $0$ at least $k$ and
we want the power at $z_{n-1}$ to be equal to $k$. Therefore, we want $l$ to be $k(n-1)$. Then it follows from the Waring formula that the absolute value of the coefficent at $z_{n-1}^k$ is equal to $(n-1)/n$. The function $f$ (as a function on $\G_n$) has only powers with degree not less than $k$ iff $k\leq n-1$. Therefore, we fix below $k=n-1$. We get the following lower estimate
$$
\kappa_{\G_n}(0;e_{n-1})\geq A_{\G_n}(0;e_{n-1})\geq \gamma_{\G_n}^{(n-1)}(0;e_{n-1})\geq\root{n-1}\of{(n-1)/n}.
$$
\end{proof}

\begin{rem}
The estimate above is better (especially asymptotically) than the general one from \cite{Nik-Pfl-Zwo} (which is $(n-1)/n$).
\end{rem}

\begin{rem}
Unfortunately, because of the form of the function $f$ above we do not have the lower estimate $\hat\gamma_{\G_n}^{(n-1)}(0;e_{n-1})$ with the same constant (with the methods from \cite{Nik-Pfl-Tho-Zwo}). Consequently, we do not get the strict inequality between $\gamma_{\G_n}(0;e_{n-1})$ and $\hat\kappa_{\G_n}(0;e_{n-1})$, $n\ge 4$.
\end{rem}

We may also improve the upper estimate for the Carath\'eodory-Reiffen pseudometric so that we shall get the inequality between the Azukawa and Carath\'eodory-Reiffen metric on the symmetrized polydisc (and therefore also between the Green function and the Carath\'eodory pseudodistance).

\begin{prop}
\label{gamgn}
Let $n\ge 3$. Then the following inequality holds
$$
\gamma_{\G_n}(0;e_{n-1})\leq \frac{1+(n/(n-2))^{n-1}}{n/(n-2)+(n/(n-2))^{n-1}}.
$$
In particular, for $n\ge 4$
$$
\gamma_{\G_n}(0;e_{n-1})< \gamma^{(n-1)}_{\G_n}(0;e_{n-1})\le A_{\G_n}(0;e_{n-1}).
$$
\end{prop}

\begin{rem}
Note that the numbers $\gamma_{\G_{n}}(0;e_{n-1})$ and $A_{\G_n}(0;e_{n-1})$ differ very distinctly asymptotically. It is elementary to see that
$$
\liminf\sb{n\to\infty}(n(1-\gamma_{\G_{n}}(0;e_{n-1}))\geq 2/(1+e^2)
$$
whereas $\lim\sb{n\to\infty} n(1-A_{\G_n}(0;e_{n-1}))=0$.
\end{rem}

\begin{proof*}{\it Proof of Theorem \ref{carazu}.}
For $n\ge 4$, this is Proposition \ref{gamgn}.
It follows from \cite[Proposition 5]{Nik-Pfl-Tho-Zwo}
that $\gamma_{\G_3}(0;e_2) < A_{\G_3}(0;e_2)$.
\end{proof*}

\begin{proof*}{\it Proof of Proposition \ref{gamgn}.}
 From
\cite[Proposition 3]{Nik-Pfl-Tho-Zwo},
for any $n\geq 3$ we have the equality $\gamma_{\G_n}(0;e_{n-1})=1/M_n$, with
$$
M_n := \inf_{a \in \C^{\mathcal P_n}}\max\{|z_{n-1}+\sum\sb{\alpha \in \mathcal P_n}a_{\alpha}z^{\alpha}|:z\in\partial\G_n\},
$$
where $\mathcal P_n$ stands for the set of all
$(n-2)$-tuples of non-negative integers $\alpha$ such that $\alpha_1+2\alpha_2+\ldots+(n-2)\alpha_{n-2}=n-1$. We proceed as in that paper; however, much more effort is required to find appropriate polynomials.

Notice that the coefficients of monic
polynomials with all zeros lying on the unit circle deliver elements $z \in\partial\G_n$, with the notation
$$
p(\lambda)= \lambda^n + \sum_{j=1}^n (-1)^j z_j \lambda^{n-j}.
$$
We shall consider two kinds of such polynomials, both with the property
that $z_j=0$, $2\le j \le n-2$.  Restricting to this subclass implies
\begin{equation}
\label{mnge}
M_n \ge \inf_{a_{(n-1,0,\dots,0)} \in \C}\max\{|z_{n-1}+a_{(n-1,0,\dots,0)}z_1^{n-1}|:(z_1,0,\dots,0,z_{n-1},z_{n})\in\partial\G_n\}.
\end{equation}
From now on we write $a=a_{(n-1,0,\dots,0)}$.

The first polynomial is
 $(\lambda^{n-1}-1)(\lambda-1)$, which gives that
 \newline $(1,0,\dots,0,(-1)^n,(-1)^{n})\in \partial\G_n$.
To find another good polynomial we need more subtle methods.
Recall that a polynomial $p(\lambda)=\sum\sb{j=0}\sp{n}a_j\lambda^j$ with $a_n\neq 0$
is called \emph{self-inversive}
if $a_{n-j}=\epsilon \bar a_j$, $j=0,\ldots,n$ for some $|\epsilon|=1$.

\begin{lemma}
\label{LaLo}
For all $n \in \Z,$ $n\ge 3$,
 all $t\in I_n:=\left[(-1)^n - \frac2{n-2},(-1)^n + \frac2{n-2}\right]$, the self-inversive polynomial
$$p_{n,t}(\lambda):=\lambda^n+(-1)^{n-1}t \lambda^{n-1}+t\lambda+(-1)^{n-1}
$$
has all its roots lying on the unit circle.
 \end{lemma}

 Then the point $((-1)^nt,0,\ldots,0,(-1)^{n-1}t,-1)$ belongs to $\partial\G_n$.

From \eqref{mnge} we see that
$$
M_n\geq \inf_{a \in \C} \max_{t \in I_n}
\left( \max \left( |(-1)^n+a 1^{n-1}| ,
|(-1)^{n-1}t+a t^{n-1}|
\right)
\right),
$$
therefore for any $t \in I_n$,
$$
M_n\geq M_n^t := \inf_{a \in \C}
\left( \max \left( |(-1)^n+a | ,
|(-1)^{n-1}t+a t^{n-1}|
\right)
\right).
$$
Since the function over which the last infimum is taken is coercive,
there exists an $a(t)\in \C$ such that
$M_n^t = \max \left( |(-1)^n+a(t) | ,
|(-1)^{n-1}t+a(t) t^{n-1}|
\right).$  Therefore
$$
(|t|^n+1) M_n^t \ge
|(-1)^nt^{n-1} +a(t) t^{n-1}| + |(-1)^{n}t-a(t) t^{n-1}|
\ge
|t^{n-1}+t|,
$$
and consequently, $M_n\geq |t^{n-1}+t|/(1+|t|^{n-1})$, for any $t \in I_n$.

Taking $t=(-1)^{n-1}(1+2/(n-2))$, we have $\gamma_n(0;e_{n-1})\leq \frac{1+(n/(n-2))^{n-1}}{n/(n-2)+(n/(n-2))^{n-1}}$.
\end{proof*}

\begin{proof*}{\it Proof of Lemma \ref{LaLo}.}

We may write that
\begin{multline*}
p_{n,t}(\lambda)\\
(\lambda+1)\big(\lambda^{n-1}-(1+(-1)^nt)\lambda^{n-2}+(1+(-1)^nt)\lambda^{n-3}+\ldots
\\
\ldots+(-1)^{n-2}(1+(-1)^nt)\lambda+(-1)^{n-1}\big)
\\
=:(\lambda+1)q_{n,t}(\lambda).
\end{multline*}
Since $q_{n,t}$ is a self-inversive polynomial we may make use of Theorem 1 of \cite{Lak-Los} (take $B=c=-d=1$) and we conclude that if $2\geq (n-2)|1+(-1)^nt|$ then all zeros of $q_{n,t}$ (and consequently all the zeros of $p_{n,t}$) lie on the unit circle as claimed.
\end{proof*}

\bibliographystyle{amsplain}

\end{document}